\documentclass{aims}
\usepackage{amsmath}
  \usepackage{paralist}
  \usepackage{graphics} 
  \usepackage{epsfig} 
\usepackage{graphicx}  \usepackage{epstopdf}
 \usepackage[colorlinks=true]{hyperref}
 \usepackage[all,cmtip]{xy}
\usepackage{amssymb}

\hypersetup{urlcolor=blue, citecolor=red}

\def\currentvolume{X}
 \def\currentissue{X}
  \def\currentyear{200X}
   \def\currentmonth{XX}
    \def\ppages{X--XX}
     \def\DOI{10.3934/xx.xx.xx.xx}

\newtheorem{theorem}{Theorem}[section]

\newtheorem{lemma}[theorem]{Lemma}
\newtheorem{proposition}{Proposition}

\theoremstyle{definition}

\newtheorem*{theorem 1}{Theorem 1}
\newtheorem*{theorem 2}{Theorem 2}

\newtheorem*{claim}{Claim}
\title[Running heading with forty characters or less] 
      {Estimates of products of norms for flows away from homoclinic tangency}

\begin{document}

\maketitle

\centerline{\scshape Qianying Xiao}

\begin{abstract}
We give a detailed proof for the estimates of products of norms for flows
away from homoclinic tangency. The estimates we can prove is weaker than the
expectation of experts.
\end{abstract}
\section{Introduction}

Assume $X\in \mathcal{X}^1(M)$ is away from homoclinic tangencies. $\phi_t$ is the flow generated by $X$, $\psi_t$ is the linear Poincar\'e flow. There exists a neighborhood $\mathcal{U}$ of $X$ such that any $Y\in \mathcal{U}$ has no homoclinic tangencies.

Wen~\cite{Wen02} shows that there exist $\mathcal{U}_2\subset \mathcal{U}$, $\gamma>0$, $N>0$, $\lambda>0$, $\delta>0$, and $T>0$, any $Y\in \mathcal{U}$, any periodic orbits $orb(x)$ of $Y$ with period $\tau>T$,
\begin{enumerate}
\item $\psi_\tau (x)$ has at most one eigenvalue with modulo in $[(1+\delta)^{-\tau},(1+\delta)^\tau]$, i.e. there exists eigenspace splitting $\mathcal{N}_x=V^s(x)+V^c(x)+V^u(x)$ such that $ V^c(x) $ is the eigenspace corresponding to the eigenvalue with modulo in $[(1+\delta)^{-\tau},(1+\delta)^\tau]$, and $\dim V^c(x)\leq1$.
\item $\angle(V^s(x),V^c(x)+V^u(x))>\gamma$, $\angle(V^s(x)+V^c(x),V^u(x))>\gamma$.
\item The splitting is dominated:
 $$\|\psi_t|V^s(x)\| \|\psi_{-t}|V^c(\phi_t(x))+V^u(\phi_t(x))\|\leq Ne^{-\lambda t},$$

$$\|\psi_t|V^s(x)+V^c(x)\| \|\psi_{-t}|V^u(\phi_t(x))\| \leq Ne^{-\lambda t}.$$
\item Estimation of norms of products:
 $$\|\psi_\tau|V^s(x)\| \leq N(1+\delta)^{-\tau},$$
$$\|\psi_{-\tau}|V^u(x)\| \leq N(1+\delta)^{-\tau}.$$

\end{enumerate}

The above estimates are improved by Wen~\cite{Wen04}. In ~\cite[Lemma 3.4]{Wen04}, Wen gives the estimates of products of norms, which are stronger than the previous result. Wen regards the proofs are standard in Liao's and Mane's work( for instance see~\cite[Page 528]{Mane}), and he doesn't give details of the proofs there. We try to prove it as an exercise, and it turns out we are not able to prove the Item 2 of ~\cite[Lemma 3.4]{Wen04}. Instead we prove a weaker version, and it works for our purpose to prove Theorem 2.1 of our work ~\cite{XZ}. The following ideas of Liao: minimal nonhyperbolic set, (local) star flows and the Selecting lemma, which are stressed by Wen extensively, help us to avoid a direct use of the non-proved estimates.

\begin{proposition}
There exists neighborhood $\mathcal{U}_3\subset \mathcal{U}_2$ of $X$, $0<\lambda_1<1$, $T^\prime>0$, any $Y\in \mathcal{U}_3$, any periodic point $x$ of $Y$ with period $\tau>\max\{T,T^\prime\}$, any partition of $[0,\tau]$:
$$0=t_0<t_1<\cdots<t_\ell=\tau$$
such that $t_{i+1}-t_i\geq T^\prime$ for $0\leq i\leq \ell -1$, the following are satisfied:
either
$$\prod^{\ell-1}_{i=0}\|\psi_{t_{i+1}-t_i}|V^s(\phi_{t_i})\|\leq N\lambda_1^\tau,$$
or
$$\prod^{\ell-1}_{i=0}\|\psi_{t_i-t_{i+1}}|V^s(\phi_{t_{i+1}})\|\leq N\lambda_1^\tau,$$
\end{proposition}

\section{The proofs}

Let us give the details of estimating the products of norms.
For the sake of completeness, let us insert a well-known fact in linear algebra.
Let $E^n$ denotes the $n$-dimensional Euclidean space. $A:E^n\rightarrow E^n$ is a linear map.

\begin{lemma}~\cite[Page 528]{Mane}
Any $\epsilon>0$, $v\in E^n$ with $\|v\|=1$, there exist $Q:E^n\rightarrow E^n$ such that $\|Q-\mathrm{id} \|\leq \epsilon$, and that $\| A \circ Q(v) \| \geq \epsilon n^{ - \frac{3}{2}} \|A\|$.

\end{lemma}

\begin{proof}
Let $\{ e_1, \cdots, e_n \} $ be an orthonormal basis, $ v=(v_1, \cdots, v_n) $. We may as well assume $ |v_1| \geq n^{- \frac{1}{2}} \|v\| $.

$ a_{ji} = ( A e_i )_j $ such that

$$ \|A\| \leq n^{ \frac{1}{2} } \|A e_i\| \leq n |a_{ji}| $$.

$Q: E^n \rightarrow E^n$ such that $ Q(e_1)=e_1+ \epsilon t e_i  $, $t=1$ or $-1$, and $Q(e_k)=e_k$ for $k>1$. Obviously $ \| Q- \mathrm{id} \| \leq \epsilon $.

$$ \|A \circ Q (v) \| \geq | (A \circ Q (v))_j |= |(A(v))_j + \epsilon t a_{ji} v_1 |.$$

Choose $t$ such that $ (A(v))_j \cdot \epsilon t a_{ji} v_1 \geq 0$, then

$$  \|A \circ Q (v) \| \geq \epsilon |a_{ji} v_1| \geq \epsilon n^{ -\frac{3}{2}} \|A \| .$$

\end{proof}

\begin{proof}[Proof of the Proposition]
Let $\epsilon >0$, $ \mathcal{U}_3 \subset \mathcal{U}_2 $ be determined by Frank's lemma. $K$ is an upperbound of $ \| \psi^Y_t \|$ for $ Y\in \mathcal{U}_3 $, $ 1\leq t \leq 2 $.

$\epsilon^\prime = \epsilon K^{-1}$, $ \epsilon_1>0$ such that
$$ \frac{1+\gamma}{\gamma}2\epsilon_1<\epsilon^\prime. $$

$$ T^\prime > \frac{\frac {3} {2} \ln n-\ln \epsilon_1} {\ln (1+\frac {1} {2} \delta)}. $$

Given any $ Y\in \mathcal{U}_3 $, any periodic point $x$ of $ Y$ with periodic $ \tau \geq \max\{ T^\prime, T\} $, any partition of $[0,\tau]$:
 $$ 0=t_0<t_1<\cdots<t_\ell=\tau, $$
such that $ t_{i+1}-t_1 \geq T^\prime$ for $0\leq i \leq \ell-1$.

\begin{itemize}

\item[Case 1]
Assume $ \dim V^c(x) = 1 $. Denote $\mathcal{N}_{\phi_{t_i}(x)}$ as $N_i$, $V^\sigma_{\phi_{t_i}(x)}$ as $V^\sigma_i$ for $\sigma =s,c,u$, $i=0,\cdots,\ell-1$.

Define
$ Q_1: N_1 \rightarrow N_1 $ such that $ Q_1|V^c_1+V^u_1 =\mathrm{id} $, $Q_1(V^s_1)=V^s_1$, $ \| Q_1|V^s_1-\mathrm{id}\| <\epsilon_1 $, $v_1\in V^s_1$ such that $ \|v_1\| = \| \psi_{t_1}|V^s(x)\|$, and
$$ \|\psi_{t_2-t_1} \circ Q_1 (v_1) \| \geq \epsilon_1 n^{-\frac{3} {2}} \|\psi_{t_2-t_1}|V^s_1 \| \| \psi_{t_1}|V^s(x)\|. $$\label{equ1}
Moreover,
$ \|Q_1 - \mathrm{id} \|< \frac{1+\gamma}{\gamma}\epsilon_1<\epsilon^\prime$.

Similarly, we can define $ Q_2: N_2 \rightarrow N_2 $ such that $ \| Q_2-\mathrm{id} \|<\epsilon^\prime$, $ Q_2|V^c_2 + V^u_2 = \mathrm{id}$, $ Q_2(V^s_2)=V^s_2$, and that
$$ \| \psi_{t_3 - t_2} \circ Q_2 \circ \psi_{t_2 -t_1} \circ Q_1 \circ \psi_{t_1}|V^s(x) \| $$
$$\geq (\epsilon_1 n^{-\frac{3}{2}})^2 \|\psi_{t_3-t_2}|V^s_1\| \| \psi_{t_2-t_1}|V^s_1 \| \|\psi_{t_1}|V^s(x).$$

$$ \cdots \cdots $$

$Q_{\ell-1}:N_{\ell-1} \rightarrow N_{\ell-1}$, such that $ \|Q_{\ell-1}-\mathrm{id} \|<\epsilon^\prime $, $Q_{\ell-1}|V^c_{\ell-1} + V^u_{\ell-1}= \mathrm{id}$, $ Q_{\ell-1}(V^s_{\ell-1})= V^s_{\ell-1}$.

\begin{align}\label{equ1}
\|(\prod^{\ell-1}_{i=1} \psi_{t_{i+1}-t_i}|V^s_i\circ Q_i) \circ \psi_{t_1}|V^s(x) \|\\
\geq (\epsilon_1 n^{-\frac{3}{2}})^{\ell-1} \prod^{\ell-1}_{i=0} \|\psi_{t_{i+1}-t_i}|V^s_i \|\\
\geq (\epsilon_1 n^{-\frac{3}{2}})^{\frac{\tau}{T^\prime}} \prod^{\ell-1}_{i=0} \|\psi_{t_{i+1}-t_i}|V^s_i \|
\end{align}

Denote $L_{i,r}=\psi_{t_{i+1}-t_i} \circ (\mathrm{id}+r (Q_i-\mathrm{id}))$ for $i=1, \cdots, \ell-1$,
$L_r=L_{\ell-1,r} \circ \cdots \circ L_{1,r} \circ L_{0,r}$ $(0\leq r \leq 1)$.

\begin{claim}
\emph{The spectral radius of $L_1|V^s(x)$ is less than $(1+\delta)^{-\tau}$}. Otherwise, for some $r$, $L_r$ admits at least two eigenvalues with modulo in $[(1+\delta)^{-\tau},(1+\delta)^\tau]$. By Frank's lemma, there exists $Z\in \mathcal{U}_2$ such that $\psi^Z_\tau(x)=L_r$, a contradiction.
\end{claim}

Apply Frank's lemma, there exist $Z^\prime \in \mathcal{U}_2$ such that $Z^\prime=Y$ along $orb(x)$, $\psi^{Z^\prime}_\tau(p)=L_1$, $V^\sigma_{Z^\prime}(x)=V^\sigma(x)$, $\sigma=s,c,u$.
Hence $$ \|L_1|V^s(x)\| \leq N(1+\delta)^{-\tau}$$
According to inequality \ref{equ1}, and by $ T^\prime > \frac{\frac {3} {2} \ln n-\ln \epsilon_1} {\ln (1+\frac {1} {2} \delta)} $,

$$\prod^{\ell-1}_{i=0} \|\psi_{t_{i+1}-t_i}|V^s_i\| \leq N[(1+\delta)^{-1}(\epsilon_1 n^{-\frac{3}{2}})]^{-\frac{1}{T^\prime}} < N( \frac{1+\frac{1}{2}\delta}{1+\delta})^\tau $$

Let $\lambda_1=\frac{1+\frac{1}{2}\delta}{1+\delta}$, and we are done.

\item[Case 2]

Assume $\dim V^c(x)=0$.
As in case 1, we have $Q_i$ for $i=1,\cdots,\ell-1$, such that $Q_i|V^u_i=\mathrm{id}$, $Q_i(V^s_i)=V^s_i$, and $\|Q_i|V^s_i-\mathrm{id}\|<\epsilon_1$. Moreover,
$$\|(\prod^{\ell-1}_{i=1} \psi_{t_{i+1}-t_i}|V^s_i\circ Q_i) \circ \psi_{t_1}|V^s(x) \|$$
$$\geq (\epsilon_1 n^{-\frac{3}{2}})^{\frac{\tau}{T^\prime}} \prod^{\ell-1}_{i=0} \|\psi_{t_{i+1}-t_i}|V^s_i \|$$

Similarly, we have $P_i$ for $i=1,\cdots,\ell-1$, such that $P_i|V^s_i=\mathrm{id}$, $P_i(V^u_i)=V^u_i$, and $\|P_i|V^u_i-\mathrm{id}\|<\epsilon_1$, and
 $$\| \psi_{-t_1} \circ P_1 \circ \cdots \circ \psi_{t_{\ell-2}-t_{\ell-1}} \circ P_{\ell-1} \circ \psi_{t_{\ell-1}-t_\ell}|V^u(x)\|$$
$$ \geq (\epsilon_1 n^{-\frac{3}{2}})^{\frac{\tau}{T^\prime}} \prod^{\ell-1}_{i=0} \|\psi_{t_i -t_{i+1}}|V^u_{i+1} \| , $$

$$ \| P_1^{-1} \circ Q_i - \mathrm{id} \| \leq \frac{1+\gamma}{\gamma}2\epsilon_1<\epsilon^\prime.$$

For $0\leq \alpha \leq 1$, $0\leq \beta \leq 1$, define $ L_{i,\alpha,\beta}=\psi_{t_{i+1}-t_i} \circ (\mathrm{id}+\alpha(P^{-1}_i-\mathrm{id})) \circ (\mathrm{id}+\beta(Q_i-\mathrm{id})).$

Denote $L_{\alpha,\beta}=L_{\ell-1,\alpha,\beta}\circ\cdots\circ L_{0,\alpha,\beta}$.

\begin{claim}
\emph{Either the spectral radius of $L_{1,1}|V^s(x)$ is less than $ (1+\delta)^{-\tau} $, or the spectral radius of $ L_{1,1}|V^u(x)  $ is greater than $ (1+\delta)^\tau $}. Otherwise some $L_{\alpha,\beta}$ admits two eigenvalues (counting multiplicity) with modulo in $ [(1+\delta)^{-\tau},(1+\delta)^\tau] $. Apply Frank's lemma, there exists $ Z\in \mathcal{U}_2$ such that $\psi^Z_\tau(x)$ has two eigenvalues with modulo in $ [(1+\delta)^{-\tau},(1+\delta)^\tau] $, a contradiction.
\end{claim}

By Frank's lemma, there exists $Z^\prime\in \mathcal{U}_2$, such that $ \psi^{Z^\prime}_\tau(x)=L_{1,1} $, therefore $\|L_{1,1}|V^s(x) \| \leq N (1+\delta)^{-\tau}$ or $\|L_{1,1}^{-1}|V^u(x) \| \leq N (1+\delta)^{-\tau}$.

Note that $ L_{1,1}|V^s(x)=( \prod^{\ell-1}_{i=0}\psi_{t_{i+1}-t_i}\circ Q_i)|V^s(x)$,
$ L_{1,1}^{-1}|V^u(x)=\psi_{-t_1} \circ \prod^1_{i=\ell-1}P_i\circ \psi_{t_i-t_{i+1}}|V^u(x)$.

Hence
$$ (\epsilon_1 n^{-\frac{3}{2}})^{\frac{\tau}{T^\prime}} \prod^{\ell-1}_{i=0} \|\psi_{t_{i+1}-t_i}|V^s_i \|\leq  N (1+\delta)^{-\tau},$$
or
$$ (\epsilon_1 n^{-\frac{3}{2}})^{\frac{\tau}{T^\prime}} \prod^{\ell-1}_{i=0} \|\psi_{t_i -t_{i+1}}|V^u_{i+1} \|\leq  N (1+\delta)^{-\tau}.$$

Let $ \lambda_1=\frac{1+\frac{1}{2}\delta}{1+\delta}$, then

$$\prod^{\ell-1}_{i=0} \|\psi_{t_{i+1}-t_i}|V^s_i \|\leq N\lambda_1^\tau,$$
or
$$\prod^{\ell-1}_{i=0} \|\psi_{t_i -t_{i+1}}|V^u_{i+1} \|\leq  N\lambda_1^\tau.$$

\end{itemize}

\end{proof}

\end{document}